\documentclass[11pt]{article}
\usepackage{amsmath,amsthm,amssymb,amscd}
\usepackage{graphicx}
\usepackage{graphics}
\usepackage{verbatim}
\usepackage{authblk}
\usepackage{mathrsfs}
\usepackage{hyperref}
\pdfoutput=1
\hypersetup{
  pdfinfo={
    Title={Your Title Here},
    Author={Your Name Here},
    Subject={If you want to put something here, do so},
    Keywords={Add some keywords if you feel so inclined}
  }
}
\usepackage{pdfpages}
\usepackage{color}
\usepackage[top=1in, bottom=1in, left=1.25in, right=1.25in]{geometry}
\usepackage{enumerate}

\usepackage{enumitem}
\usepackage{paralist}

\newtheorem{theorem}{Theorem}[section]

\newtheorem{lemma}[theorem]{Lemma}
\newtheorem{definition}[theorem]{Definition}
\newtheorem{corollary}[theorem]{Corollary}

\newtheorem{conjecture}[theorem]{Conjecture}

\def\keywords{\vspace{.5em}
{\textit{Keywords}:\,\relax%
}}

\begin{document}

\title{Hamiltonian cycles in annular decomposable Barnette graphs}

\author{Saptarshi Bej\thanks{{Author for correspondence:            saptarshibej24@gmail.com, saptarshi.bej@uni-rostock.de}
	} }
\affil{Institute of Computer Science,
		University of Rostock, Germany}

\date{}

\maketitle

\begin{abstract}
Barnette's conjecture is an unsolved problem in graph theory. The problem states that every $3$-regular (cubic), $3$-connected, planar, bipartite (Barnette) graph is Hamiltonian. Partial results have been derived with restrictions on number of vertices, several properties of face-partitions and dual graphs of Barnette graphs while some studies focus just on structural characterizations of Barnette graphs. Noting that Spider web graphs are a subclass of \textit{Annular Decomposable Barnette} (ADB graphs) graphs and are Hamiltonian, we study ADB graphs and their annular-connected subclass (ADB-AC graphs). We show that ADB-AC graphs can be generated from the smallest Barnette graph ($B_0$) using recursive edge operations. We derive several conditions assuring the existence of Hamiltonian cycles in ADB-AC graphs without imposing restrictions on number of vertices, face size or any other constraints on the face partitions. We show that there can be two types of annuli in ADB-AC graphs, \textit{ring annuli} and \textit{block annuli}. Our main result is, ADB-AC graphs having non singular sequences of ring annuli are Hamiltonian.   
\end{abstract}

 
\section{Introduction}
Barnette's conjecture is a long standing unsolved problem in mathematics proposed by David W. Barnette in 1969. The conjecture states:
\begin{conjecture}\label{barcon}
Every $3$-regular, $3$-connected planar bipartite graph is Hamiltonian. We will henceforth use the term Barnette graphs to refer to $3$-regular, $3$-connected planar bipartite graphs.
\end{conjecture}
The development towards the Barnette's conjecture started as early as 1880, when P.G. Tait proposed a weaker statement:
\begin{conjecture}\label{taitcon}
Every $3$-regular, $3$-connected planar graph is Hamiltonian.
\end{conjecture}
This conjecture was ultimately disproved by W.T. Tutte, when he constructed a couter-example to the conjecture, a graph built on $46$ vertices \cite{tut_counter}. With further attempts to find a smaller counterexamples for Tait's conjecture, Holton and McKay found a counterexample on $38$ vertices \cite{smallcounter}. However, none of these counterexamples found were bipartite. Tutte thus proposed a conjecture stating:
\begin{conjecture}\label{tuttcon}
Every $3$-regular, $3$-connected bipartite graph is Hamiltonian.
\end{conjecture}
Interestingly, a counterexample of this conjecture was found in 1976 by J.D. Horton which is a graph on $96$ vertices, known as the Horton graph \cite{HORTON}. The Barnette conjecture was thus proposed by combining Conjecture \ref{taitcon} and Conjecture \ref{tuttcon} and remains unsolved.\par
There are several studies that addressed Conjecture \ref{barcon}. There have been attempts to understand the structure of Barnette graphs in general \cite{Florek, Florek3}. One of the well known approaches is the proof by Holton, Manvel and McKay in 1984, showing that all $3$-regular, $3$-connected planar bipartite graphs with no more than 64 vertices are Hamiltonian \cite{alpha}. In the same paper, they also present an interesting approach to generate all possible Barnette graphs from the smallest Barnette graphs, using simple operations defined by adding small structures to existing graphs. Other attempts have been made by studying the duals of Barnette graphs \cite{LU, FLOREK2, bagh}. We conclude from the existing literature, that most of the studies, aiming to find Hamiltonian subclasses of Barnette graphs, focus on studying the dual graphs or imposing restriction on the faces of Barnette graphs.\par
In this article, we study properties of \textit{Annular Decomposable Barnette} (ADB graphs) graphs, a sublclass of Barnette graphs. The idea of annular graphs is not novel in itself \cite{annular1, annular2}. Another article derives necessary and sufficient conditions when cubic plane graphs have a rectangular-radial drawing \cite{rect_rad}. Interestingly, ADB graphs can be thought of as Barnette graphs that have a rectangular-radial drawing. One related work, that inspired us particularly to study ADB graphs in context to a Hamiltonian problem, is the result that Spider-web graphs have a Hamiltonian cycle excluding any edge in such graphs \cite{spider-web}. The Spider-web graphs can be characterised as ADB-graphs(see Figure 2) in \cite{spider-web}) and thus, can be viewed as a subclass of graphs we have studied. We show how \textit{Annular Decomposable Barnette Annular Connected} (ADB-AC) graphs, a certain subclass of ADB-graphs can be built recursively from the smallest Barnette graph using edge operations. We derive some sufficient conditions under which, ADB-AC graphs are Hamiltonian. In our results we do not impose restrictions on number of vertices, face size or any other constraint on the face partitions. We show that there can be two types of annuli in ADB-AC graphs, \textit{ring annuli} and \textit{block annuli}. Our main result is, ADB-AC graphs having non singular sequences of ring annuli are Hamiltonian.

\section{ Preliminary definitions and notations}
As per the Jordan curve theorem, a non-self-intersecting continuous loop in the 2-D plane has an interior and an exterior. For a planar embedding of a Barnette graph we can view a cycle or face, $C$ in $G$ as a Jordan curve and naturally induce a sense of interior, exterior and a boundary to $C$. Note that a Hamiltonian cycle $H$, if it exists, in a planar embedding of a Barnette graph $G$ would have the following properties:
\begin{enumerate}
    \item The interior of $H$ is path connected
    \item The exterior of $H$ is path connected
    \item All vertices in $G$ lie in the boundary of $H$. 
\end{enumerate}
The idea is illustrated in Figure \ref{demo1}(a).\par

\begin{figure}[ht] 
\centering
\includegraphics[scale=.8]{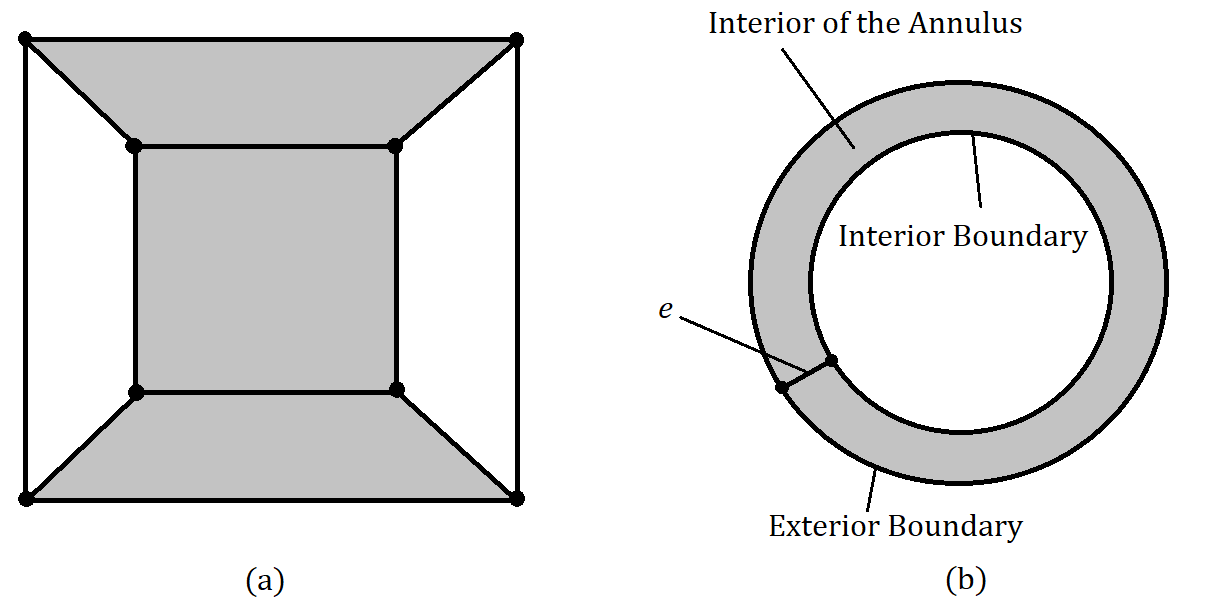}
\caption{(a) A Hamiltonian cycle in the smallest Barnette graph $B_0$ (b) An annulus is defined as two non-intersecting Jordan curves on a 2-D plane, topologically equivalent to two non-intersecting concentric circles. The figure shows an annulus and an edge $e$, lying in the interior of an annulus.}   \label{demo1}
\end{figure}

Note that, two non intersecting Jordan curves on a 2-D plane can always be drawn on the plane such that they are topologically equivalent to two non-intersecting concentric circles. We will refer to such a structure as an \textit{annulus} (See Figure \ref{demo1}(b)). Intuitively, an annulus will have an interior and the exterior boundary once drawn on a plane as shown in Figure \ref{demo1}(b). The enclosed region between the two Jordan curves in a given planar projection of an annulus will be referred to as the \textit{interior of the annulus}. 
\begin{definition}\label{ADB}
We call a Barnette graph $G$ to be Annular Decomposable, if there exists a planar embedding of $G$ such that:
\begin{enumerate}
    \item $G$ can be partitioned into sequence of annuli $A_1,\dots,A_x$
    \item The interior boundary of $A_1$ is a face called interior face $f_i$
    \item The exterior boundary of $A_x$ is a face called exterior face $f_e$
    \item Interior boundary of some annulus $A_i$, $x \geq i>1$, is the exterior boundary for some annulus $A_{i-1}$
    \item Every edge in $v_iv_j\in E(G)$ are such that, either $v_i$ and $v_j$ lie on the boundary of some annulus, or $v_i$ lies in outer boundary of some annulus $A_k$ and $v_j$ lies on the inner boundary of $A_k$ ($1\leq k \leq x$). 
\end{enumerate}
If $G$ can be partitioned into $x$ annuli, we call $G$ an $x$-Annular Decomposable Barnette ($x$-ADB graph) graph.
\end{definition}

\begin{definition}
Let $G$ be a $x$-ADB graph. If in a planar embedding of $G$, $A_1,\dots,A_x$, $f_i$ and $f_e$ are realized, we call it a planar annular embedding of $G$.
\end{definition}

\begin{definition}
Let $G$ be a $x$-ADB graph, $x>2$. If for every annulus $A_k$, $1\leq k \leq x$, deleting all edges (along with respective vertices) in the interior of $A_k$ produces either two separate Barnette graphs or a Barnette graph and a vertex-less Jordan curve, then we call $G$ an ADB annular connected ($x$-ADB-AC graph) graph.
\end{definition}
In this article we prove that ADB-AC graphs having non singular sequences of ring annuli are Hamiltonian.

\begin{lemma}\label{existADBNAC}
There exists ADB graphs that are not annular-connected.
\end{lemma}
\begin{proof}
Let $G$ be an $n$-ADB graph. Then $H$ be a copy of $G$. Consider two consecutive edges $e_G$ and $e_G'$ that lie in the interior of the outer annuals of $G$, that are adjacent to vertices $v_G^1$ and $v_G^2$ of $G$. Consider the two analogous edges $e_H$ and $e_H'$ in $H$, that are adjacent to vertices $v_H^1$ and $v_H^2$ of $G$. Let $f_G$ and $f_H$ be the edges on the outer boundary of the outer annuli of $G$ and $H$, connecting $v_G^1$ with $v_G^2$ and $v_H^1$ with $v_H^2$ respectively. Note that each of the four vertices, $v_G^1$, $v_G^2$, $v_H^1$ and $v_H^2$, there is yet another edge adjacent to the vertices which we will denote as, $e_{v_G^1}$, $e_{v_G^2}$, $e_{v_H^1}$ and $e_{v_H^2}$ respectively. We now form a new graph $J$ by the following steps:
\begin{enumerate}
    \item Join $f_G$ and $f_H$ by drawing two edges across the existing edges
    \item Join $e_{v_G^1}$ and $e_{v_H^1}$ by drawing one edge across the existing edges
    \item Join $e_{v_G^2}$ and $e_{v_H^2}$ by drawing one edge across the existing edges
\end{enumerate}
The resulting graph $K$ has $2x+1$ annuli and it is easy to see that $K$ is an ADB graph, where $x$ is the number of annuli in $G$. Note that $G$ can be viewed as a subgraph of $K$. Now we delete all the edges in the interior of the outer annulus of $G$, to obtain two graphs $K'$ and $K''$ with $x+1$ and $x-1$ annuli respectively. Note that, $K'$ is not a Barnette graph since it is not bipartite. (See Figure \ref{counter})
\end{proof}

\begin{figure}
\centering
\includegraphics[scale=.7]{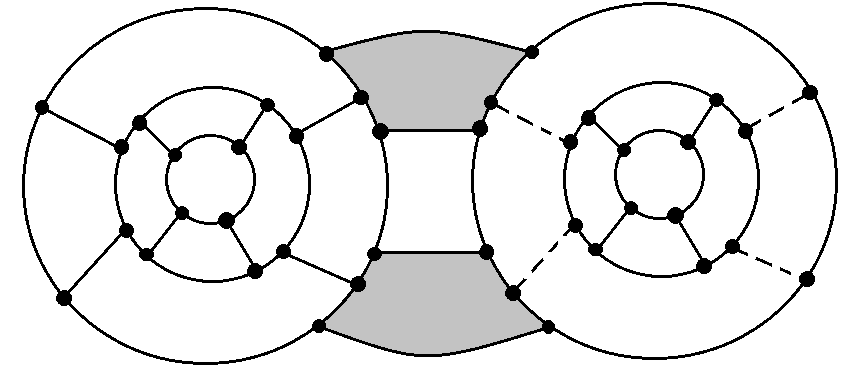}
\caption{The graph shown in the figure is an ADB graph. Note that if we delete the edges in the interior of the outer annulus of the right hand side graph as shown by the dotted lines then the graph is decomposed into a Barnette AD graph (left hand side) and a non-Barnette AD graph (right hand side). The 5-cycles are marked in grey.}   \label{counter}
\end{figure}

\begin{definition}
Let $G$ be ADB-AC graph drawn on a plane. We denote the outer annulus of $G$ as $O_G$.
\end{definition}

\begin{definition}
Let $G$ be ADB-AC graph drawn on a plane. Let $G$ have $x$ annuli. For some $x>k>0$, we define the graph formed by considering restricting $G$ only until the $k$-th annulus, to be the $k$-annular restriction of $G$, denoted by $G_k$. Note that we clarify here that in $O_{G_k}$, we do not consider any vertex adjacent to the edges in the interior of $O_{G_{k+1}}$.

\end{definition}

\section{Main Results}

\begin{lemma}\label{edgecon}
Let $G$ be a Barnette graph. Then $G$ is 2-edge connected.
\end{lemma}
\begin{proof}
\begin{figure}[ht] 
\centering
\includegraphics[scale=.8]{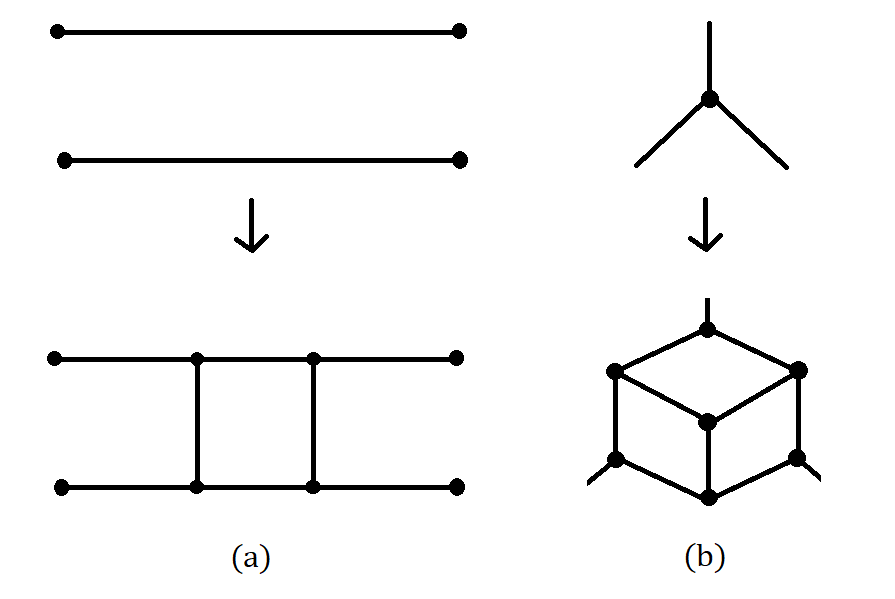}
\caption{Figure showing operations that can be used to build all possible Barnette graphs from the smallest Barnette graph shown in Figure \ref{demo1}-(a), proved by Holton, Manvel and McKay \cite{alpha}.}   \label{demo2}
\end{figure}
Holton, Manvel and McKay proved that every Barnette graph can be constructed from the smallest Barnette graph $B_0$ (shown in Figure \ref{demo1}-(a)), by using the operations shown in Figure \ref{demo2}. Given that $B_0$ is $2$-edge connected, it is easy to see by induction that any Barnette graph $G$, would be $2$-edge connected.
\end{proof}
Note that given a Hamiltonian Barnette Graph $G$, if we apply the operation described in Figure \ref{demo2}(b), which we will henceforth call a $\beta$-operation, any arbitrary number of times, the result is still a Hamiltonian graph. However, this is not true if we apply the operation described in Figure \ref{demo2}(a), which we will henceforth call a $\alpha$-operation, any arbitrary number of times. Later, in Theorem \ref{construct}, we show that ADB-AC graphs can be constructed from $B_0$ by using arbitrary number of $\alpha$-operations only.

\begin{lemma}\label{Hamil_four}
If every annulus of an ADB graph $G$ has a $4$-Cycle, then $G$ is Hamiltonian. Note that this statement is true for any AD-Planar graph.
\end{lemma}
\begin{proof}
The proof of the Lemma is clear from Figure \ref{hamil}. Note that in Figure \ref{hamil}, the illustrated graph has an odd number of annuli. If we consider a graph with even number of annuli, the same strategy will still work.  
\end{proof}
\begin{figure}[ht] 
\centering
\includegraphics[scale=.3]{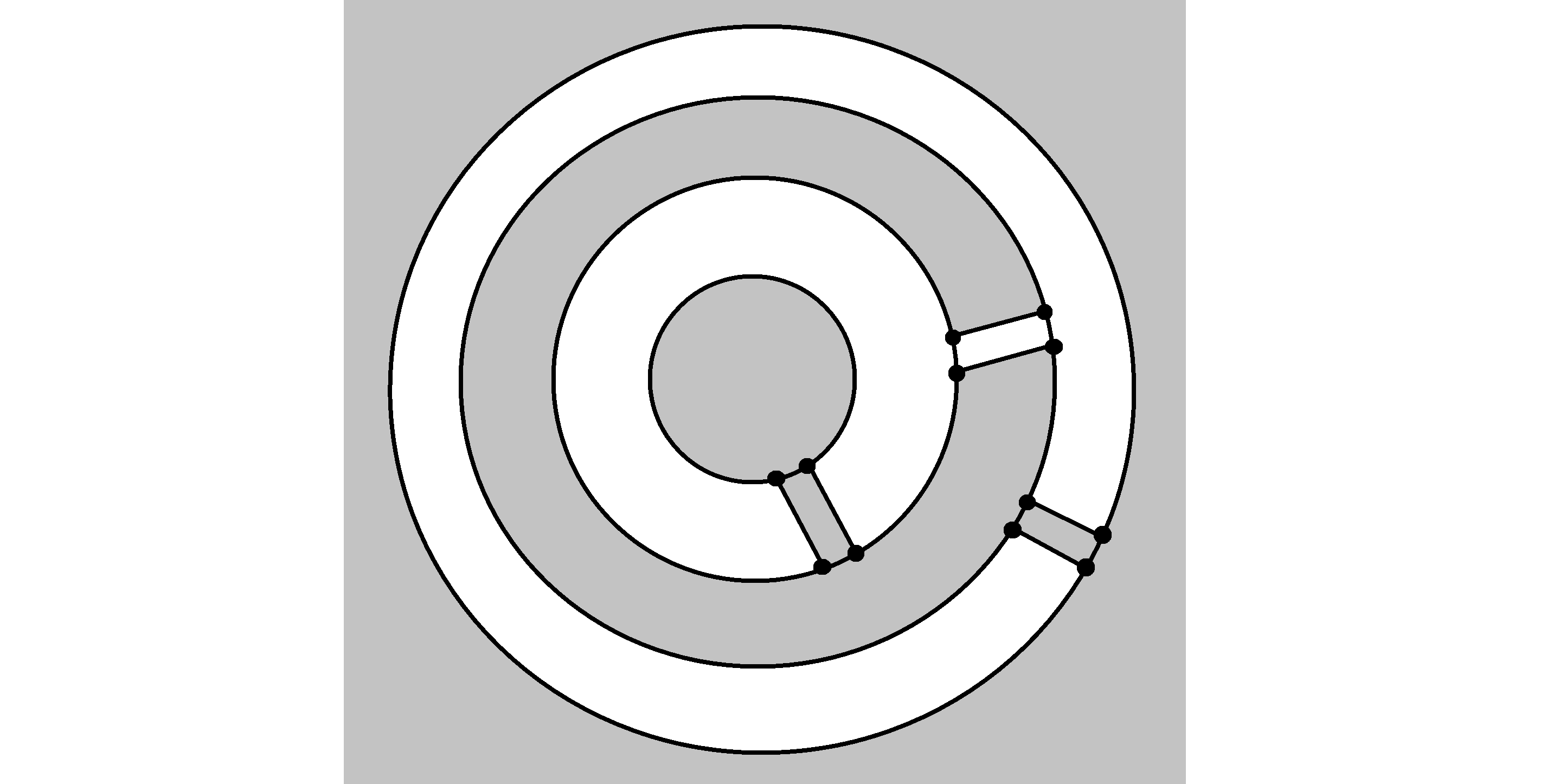}
\caption{Figure showing strategy of constructing a Hamiltonian cycle given an ADB graph with a $4$-cycle in every annulus.}\label{hamil}
\end{figure}

Let us consider an ADB-AC graph $G$. Let us consider the outer annulus of the graph $G$. There are even number of faces in the outer annulus of $G$ since, the outer boundary of the outer annulus of $G$, is an even cycle. Thus, the set of faces in the outer annulus of $G$, denoted by $F(O_G)$ can be partitioned into subsets $F^1(O_G)$ and $F^2(O_G)$, such that no two faces in $F^i(O_G)$, $i \in \{1,2\}$ are adjacent to each other. 
\begin{lemma}\label{face_part}
Let us consider an ADB-AC graph $G$ with $x \geq 3$ annuli. Either $F^1(O_G)$ or $F^2(O_G)$ consists of only $4$-cycles.
\end{lemma}
\begin{proof}
Let us create a graph $H$, by deleting all the edges in the interior of the outer annulus (along with their respective vertices). $H$ would be an ADB-AC graph, since $G$ is an ADB-AC graph. Thus, $F(O_H)$ can be partitioned into subsets $F^1(O_H)$ and $F^2(O_H)$, such that no two faces in $F^i(O_H)$, $i \in \{1,2\}$ are adjacent to each other.\par
We claim that, the edges in the interior of the outer annulus of $O_G$ are all attached to the outer boundary of faces in $O_H$, such that all such faces belong either in $F^1(O_H)$ or in $F^2(O_H)$, but not both. To prove this claim, let us assume otherwise. Then, there exits two edges $e_1$ and $e_2$ in the interior of $O_G$, such that $e_1$ is attached to the outer boundary of some face $f^1\in F^1(O_H)$ and $e^2$ is attached to some face $f^2\in F^2(O_H)$. Let, $v_1$ and $v_2$ be the two vertices by which $e_1$ and $e_2$ are attached to the outer boundary of $O_H$. Then there must be an odd number of vertices between $v_1$ and $v_2$, considering the fact that $f^1$ and $f^2$ are in different partitions. Thus, there is an odd cycle in $G$, which is impossible, since $G$ is bipartite.\par

Let us assume, without loss of generality, that all edges in interior of $O_G$ are attached to outer boundary of faces in $F^1(O_H)$. Note that for an arbitrary face $f \in F^1(O_H)$, to which some of such edges are attached, we can assert that there are only even number of edges attached to $f$, Otherwise, $f$ would transform into a odd cycle (face) $f'$ in $G$. Thus, there would always be a sequence of odd number of $4$-cycles, $c_1, \dots , c_{2k+1}$ (for some integer $k\geq 1$) formed in $O_G$ adjacent to $f'$. Starting from $c_1$, if we consider all alternate faces in $O_G$, they will all be $4$-cycles and will form a face partition of $4$-cycles in $O_G$. Thus the statement holds.    
\end{proof}

As per convention, we will assume $F^1(O_G)$ to be the face partition containing $4$-cycles only in $O_G$. Note here that, both the inner and outer annulus in any planar annular embedding of an $2$-ADB graph $G$ must have at least two $4$-cycles each. This along with Lemma \ref{Hamil_four} gives the following corollary.
\begin{corollary}
Every $2$-ADB graph is Hamiltonian.
\end{corollary}

\begin{figure}[ht] 
\centering
\includegraphics[scale=.8]{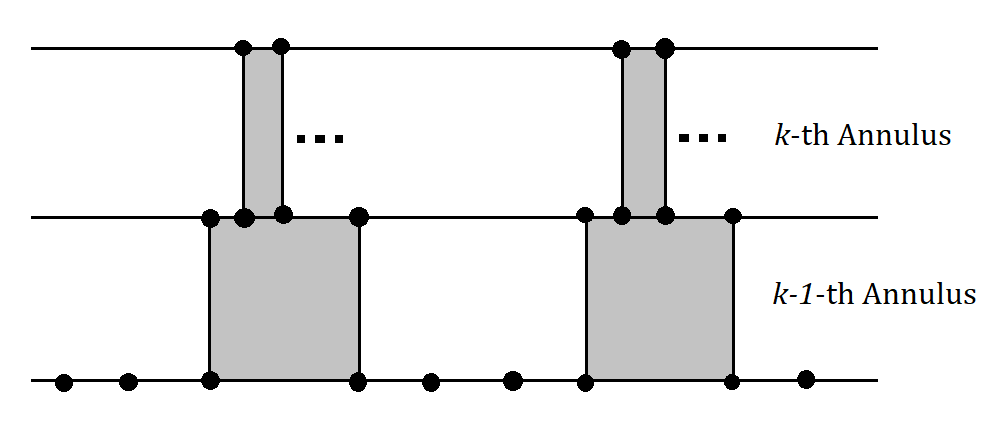}
\caption{Figure showing the recursive pyramid strategy to find a Hamiltonian cycle for Statement 1 in Theorem \ref{twocasethm}. The grey region is the path connected region that is extended recursively through the annuli. The dots show possibility of other $4$-faces in the $k$-th annular restriction.}   \label{strategyt1}
\end{figure}

\begin{corollary}
Every $G$ be an $x$-ADB graph. If there exists a planar annular projection of $G$ such that the $k$-th annulus $2 \leq k \leq x$ has more edges in its interior than the number of faces in the interior of the $k-1$th annulus, then $G$ is Hamiltonian. 
\end{corollary}
\begin{proof}
The proof of this statement also follows from Lemma \ref{Hamil_four}. Note that if there exists such a graph $G$, then each of its annulus will have a $4$-cycle.
\end{proof}

\begin{theorem}\label{construct}
Let us consider an ADB-AC graph $G$. Then $G$ can be constructed from $B_0$ by using arbitrary number of $\alpha$-operations only. 
\end{theorem}
\begin{proof}
Let $G$ have $x$ annuli for a particular planar annular embedding. For $x=1$, the result is obviously true.\par

Recall that, as per our convention, the $4$-face-only partition in the outer annulus of an ADB-AC graph $G'$ is assumed to be $F^1(O_{G'})$. Let $x>1$. By Lemma \ref{face_part},say $F^1(O_G)$ consists of only $4$-cycles. Let us select any two such $4$-cycles $m$ and $m'$. All other edges in $O_G$ can be removed by performing a inverse of $alpha$-operations. Then, we perform an inverse $\alpha$-operation on the edges $e_1$ and $e_2$ in $m$ such that, $e_1$ and $e_2$ are not in the interior of $O_G$. Now, we perform another inverse of $ \alpha$-operation by deleting edges $e'_1$ and $e'_2$ in $m'$ such that, $e'_1$ and $e'_2$ were in the interior of $O_G$ in $G$. This gives us the $(x-1)$-annular restriction of $G$. By continuing this process we can ultimately reduce the graph to $1$-annular restriction of $G$. Thus it is clear that $G$ can be constructed from $B_0$ by using arbitrary number of $\alpha$-operations only. 
\end{proof}

\begin{theorem}\label{twocasethm}
Let us consider an ADB-AC graph $G$, with $x \geq 3$ annuli.
\begin{enumerate}
    \item For any $k$, $2\leq k \leq x$, if the edges interior to $k$-th annulus in $G$ are attached to faces that are in the $4$-cycle-only face partition in $O(H_{k-1})$, where $H$ is the $k-1$ annular restriction of $G$ , then $G$ is Hamiltonian.
    \item For any $k$, $2\leq k \leq x$, if the edges interior to $k$-th annulus in $G$ are attached to faces that are not in $4$-cycle-only face partition in $O(H_{k-1})$, where $H$ is the $k-1$ annular restriction of $G$, then $G$ is Hamiltonian.
\end{enumerate}
\end{theorem}
\begin{proof}

For the first statement, it is easy to construct the Hamiltonian cycle in $G$, in a recursive manner. We start by a planar annular projection of $G$. We include the interior face $f_i$ in the interior of a path connected region $R$. Recalling our convention that, the $4$-face-only partition in the outer annulus of an ADB-AC graph $G'$ is assumed to be $F^1(O_{G'})$, we can describe this strategy by extending $H$ through $k$-th annulus by coloring the faces in $F^1(O_{G_{k}})$ in grey, for every $k$-recursively. Thus we gradually extend $R$ following the strategy shown in Figure \ref{strategyt1}. It is easy to see that, when we recursively apply this strategy for all the annuli in $G$, the boundary of the final path connected region would manifest into a Hamiltonian cycle. We will call this strategy of constructing Hamiltonian cycles henceforth as the \textit{pyramid strategy}. \par

Recall that, as per our convention, the $4$-face-only partition in the outer annulus of an ADB-AC graph $G'$ is assumed to be $F^1(O_{G'})$. Let $x>1$. For the second statement, note that for any $k$, $2\leq k \leq x$, if the edges interior to $k$-th annulus in $G$ are attached to faces that are not in $4$-cycle-only face partition in $O(H_{k-1})$, $F^2(O_H)$, where $H$ is the $k-1$ annular restriction of $G$, then for the $k$-th annulus, there is a $4$-cycle. The reason is no edges in the $k+1$-th annulus are attached to the face partition with $4$-cycles of the $k$-th annulus. Thus by Lemma \ref{Hamil_four}, $G$ has to be Hamiltonian. The Hamiltonian cycle can be constructed using the same strategy as depicted in Figure \ref{Hamil_four}. We will call this strategy of constructing Hamiltonian cycles henceforth as the \textit{ring strategy}.   
\end{proof}

Let $G$ be ADB-AC graph drawn on a plane. From Theorem \ref{twocasethm}, it is evident that the annuli of $G$ can be one of two possible types depending on how the edges in the interior of a given annulus are attached to the previous annulus. For the $k$-th annulus $A$, in $G$ if the edges in the interior of the annulus $A$ are attached to the $k-1$-th annulus such that the edges are attached to the faces in the face partition of $O_{G_{k-1}}$ consisting of only $4$-cycles, then we call $A$ a \textit{block} annulus, otherwise we call $A$ a \textit{ring} annulus. We can thus restate Theorem \ref{twocasethm} as: Given an ADB-AC graph $G$, if $G$ consists of only block annuli or only ring annuli, then $G$ is Hamiltonian. Note that the spider-web graphs that motivated the work are a subclass of ADB graphs that consist of only block annuli. Also note that given a planar representation ADB-AC graph $G$, the first and second annulus can be considered both a block annulus and a ring annulus. By convention, we will consider them as of the same category as the third annulus.
\begin{corollary}
Let $G$ be an ADB-AC graph consisting only of ring annuli, $A_1,\dots,A_x$. Let $|F_A|$ denote the number of faces in an annulus. Then $G$ must have at least $\frac{1}{2^x}\Pi_{i=1}^x(|F_{A_{i}}|)$ Hamiltonian cycles. 
\end{corollary}
\begin{proof}
Let $G$ be an ADB-AC graph consisting only of ring annuli $A_1,\dots,A_x$. Any arbitrary annulus $A_i$ must have $|F_{A_{i}}|$ $4$-cycles. One $4$-cycle from every annuli can be used to create a Hamiltonian cycle. Thus, it is clear that $G$ has at least $\frac{1}{2^x}\Pi_{i=1}^x(|F_{A_{i}}|)$ Hamiltonian cycles. Given that $|F_{A_{i}}|$ is at least $4$, we can derive that any such graph will have at least $2^x$ Hamiltonian cycles.   
\end{proof}

In Theorem \ref{twocasethm}, we have already discussed a strategy of finding Hamiltonian cycles in an ADB-AC graph $G$, given that $G$ consists of only block annuli. Now we will explore some more strategies to find a Hamiltonian cycle in $G$. 

\begin{lemma}\label{strategy2}
Let us consider an ADB-AC graph $G$, with $x \geq 3$ annuli, containing only block annuli. There are at least two more strategies for finding Hamiltonian cycles in $G$, in addition to the one described in Theorem \ref{twocasethm}.
\end{lemma}
\begin{proof}
Recall that, as per our convention, the $4$-face-only partition in the outer annulus of an ADB-AC graph $G'$ is assumed to be $F^1(O_{G'})$. Let $x>1$. Let us consider a planar projection of $G$ such that $f_i$ and $f_o$ are fixed and well defined. Note that such a graph $G$ must contain a specific substructure described as follows. This substructure contains one face from every annulus such that the faces are stacked on each other ($x$ faces stacked on one another). Let us denote these faces as $g_1, \dots, g_x$ and observe that $g_k \in F^1(O_{G_{k}})$ for the $k$ annular restriction of $G$, $3\leq k \leq x$. Also the face $g_x$ in the $x$-th annulus is a $4$-cycle. Existence of such a structure is ensured due to the fact that $G$ has only block annuli. We call such a structure a \textit{pyramid}. Now, note that, since $G$ is $2$-edge connected, there has to be at least two pyramids in $G$. We denote them as $B_1$ and $B_2$.\par

Note that, in the first annulus of $G$, there must exist a partition of faces with all $4$-cycles. We choose one from them and call it $m'$. Our strategy to build a second Hamiltonian cycle in $G$ is based on these pyramids and the $4$-cycle $m'$. The strategy is as follows:

\begin{figure}[ht] 
\centering
\includegraphics[scale=.8]{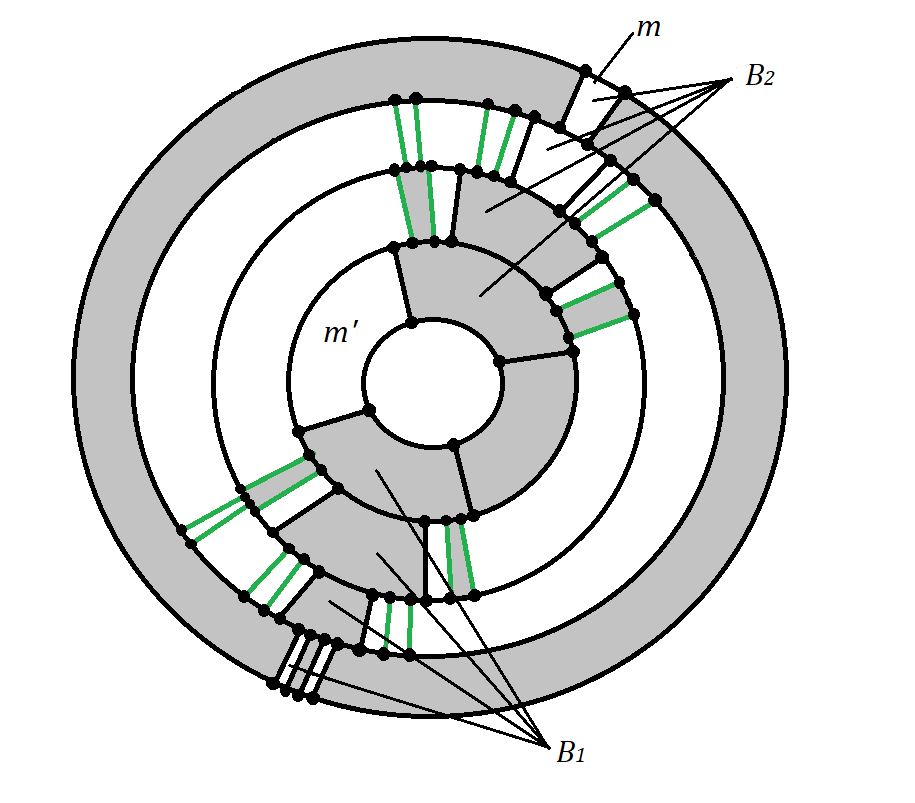}
\caption{Figure showing the pyramid-ring strategy of finding Hamiltonian cycle in ADBAC graph $G$ with all block annuli.}   \label{strategyt2}
\end{figure}

\begin{enumerate}
    \item Build a Hamiltonian cycle $H_1$ in $G$ using the strategy mentioned in Theorem \ref{twocasethm}. For the sake of understanding let us assume that we color the interior of the Hamiltonian cycle in grey as shown in Figure \ref{strategyt1}. 
    \item Note that all the faces in $B_1$ and $B_2$ are colored in grey, i.e. are interior to the cycle. Also, $m'$ is colored in white, i.e. is exterior to the cycle.
    \item Remove $f_i$ from $H$, i.e color $f_i$ in white.
    \item Color all faces in annulus $1$ except for $m'$. Call the new grey region $H_2$. 
    \item Remove coloring from all faces in annulus $x-1$ except for the face in annulus $x-1$ that belongs in $B_1$.
    \item Color all faces in annulus $x$ except for the $4$-cycles attached to the face $h_1$, such that $h_1\in B_1 \cap O(G_{x-1})$ and any one $4$-cycle $m$ attached to the face $h_2$, such that $h_2\in B_2 \cap O(G_{x-1})$. Also, we choose $m$ such that the $m \in F^1(O_G)$.
    \item Among the uncolored $4$-cycles mentioned in the step above, color the ones that are in $F^2(O_G)$.
\end{enumerate}
After Step $3$, the grey region gets path-disconnected but retains path-connectivity after Step $4$. Step $4$ also gives a Hamiltonian cycle $H_2$ in $G$. After Step $5$, note that only some vertices in the outer boundary of the $x-1$-th annulus are in the interior of the white region. After Step $6$ these vertices again lie in the boundary of the grey-white region. Note both the grey and white regions still remain path connected (the existence of $m$ ensures the path connectedness of the white region). After Step $7$ all the vertices lie in the boundary of the grey-white coloring and both the white and the grey region remain path connected. Thus the boundary of the grey-white path connected regions is a new Hamiltonian cycle $H_3$. We will call this strategy of constructing a Hamiltonian cycle in $G$ as pyramid-ring strategy or pr-strategy. The pr-strategy have been demonstrated in Figure \ref{strategyt2}.  
\end{proof}

\begin{theorem}\label{blockring}
Let us consider a planar annular-embedding of an ADB-AC graph $G$, with $x+y$ annuli, containing a sequence of $x \geq 1$ block annuli followed by a sequence of $y \geq 1$ ring annuli. Then $G$ is Hamiltonian.
\end{theorem}
\begin{proof}
Recall that, as per our convention, the $4$-face-only partition in the outer annulus of an ADB-AC graph $G'$ is termed as $F^1(O_{G'})$. When $x=1,2$ we treat all the annuli as ring annuli. For $x \geq 3$ block annuli, we first employ the pr-strategy on $G_x$. Since the $x+1$-th annulus and all annuli afterwards are ring annuli, there must exist a $4$-cycle in the $x+1$-th annulus such that the edges of the $4$ face are attached to some face $f \in F^2(O_{G_{x}})$. Since all the faces in $F^2(O_{G_{x}})$ are marked in grey as per the pr-strategy, we can extend the grey color to $f$ and follow the ring strategy for the rest of the $y-1$ annuli after that to obtain a Hamiltonian cycle (whose interior is marked in grey).
\end{proof}

\begin{theorem}\label{ringblock}
Let us consider a planar annular embedding of an ADB-AC graph $G$, with $x+y$ annuli, containing a sequence of $x \geq 1$ ring annuli followed by a sequence of $y \geq 1$ block annuli. Then $G$ is Hamiltonian.
\end{theorem}
\begin{proof}
If $x=1,2$, then we can consider all the annuli as block annuli and use the pr-strategy for finding a Hamiltonian cycle. 
Recall that, as per our convention, the $4$-face-only partition in the outer annulus of an ADB-AC graph $G'$ is termed as $F^1(O_{G'})$. If $x>2$, first we can create a Hamiltonian cycle in $G_{x-1}$ using the ring strategy, and color it's interior grey such that only one $4$-face in $O(G_{x-1})$ is colored in white. Until now, for $x$-th annulus in $G_{x-1}$, the $4$-face-only face partition has been conventionally $F^1(O_{G_{x-1}})$. We alter this convention from this annulus and call them $F^2(O_{G_{x-1}})$. Note that, this allows us to extend the grey region (Hamiltonian cycle) from $x-1$th annulus to the $x+y$th annulus, using the pr-strategy as described in Lemma \ref{strategy2}.   
\end{proof}

From Theorems \ref{blockring} and \ref{ringblock}, it is clear that:
\begin{theorem}\label{arbit}
Every ADB-AC graph $G$, with arbitrary sequence of ring and block annuli is Hamiltonian, if every sequence ring annuli that lies between two sequences of block annuli, has length at least $2$, that is, there are only non singular sequences of ring annuli in $G$.  
\end{theorem}

\subsection*{Open questions}
We propose the following open questions from our study:
\begin{enumerate}
    \item Are all ADB-AC graphs Hamiltonian?
    \item Are all ADB graphs Hamiltonian?
\end{enumerate}

\subsection*{Acknowledgements}
I thank my PhD supervisor Olaf Wolkenhauer for his constant encouragement and unfaltering support for this research. This work was in part supported by funds from Bioinformatics Infrastructure (de.NBI) and Establishment of Systems Medicine Consortium in Germany e:Med , as well as the German Federal Ministry for Education and Research (BMBF) programs (FKZ 01ZX1709C). I thank Miss Shukla Sarkar for her patient support.

\bibliography{Barnette.bib}
\bibliographystyle{plain}

\end{document}